\documentclass[final,onetabnum,onefignum]{siamltex}

\usepackage{amsmath,amssymb,mathtools}
\usepackage{array}
\usepackage{booktabs}
\usepackage{graphicx}
\graphicspath{{figures/}}

\providecommand{\texorpdfstring}[2]{#1}
\newtheorem{assumption}[theorem]{Assumption}
\newtheorem{construction}[theorem]{Construction}
\newtheorem{maintheorem}{Theorem}[section]
\newtheorem{remark}[theorem]{Remark}

\title{Random-Feature Kalman Filtering for Linear PDE
Data Assimilation}

\author{
Xi'an Li\thanks{Ceyear Technologies Co., Ltd, Qingdao 266555, China.}
\and
Jiale Linghu\thanks{School of Mathematics and Statistics, Xidian University, Xi'an, China.}
\and
Yangshuai Wang\thanks{Corresponding author. Department of Mathematics, National University of Singapore, Singapore (\texttt{yswang@nus.edu.sg}).}
}

\begin{document}

\maketitle

\begin{abstract}
Data assimilation for time-dependent partial differential equations
(PDEs) requires Bayesian updates of an evolving field from streaming,
sparse, and noisy observations, while keeping the filtering state finite
dimensional. We introduce a random-feature Kalman filtering framework for
linear PDE data assimilation. Once the random features are frozen and the
linear PDE is Galerkin discretized, the coefficient vector satisfies a
finite-dimensional linear-Gaussian state-space model, so the Kalman
recursion gives the exact posterior for the chosen coefficient model. For
non-orthogonal random-feature draws, we construct a mass-whitened
effective-rank coordinate system that removes near-null mass directions and
identifies the posterior dimension \(r\). For the heat equation with implicit-Euler time stepping, we prove a
high-probability posterior-contraction and PDE-consistency theorem in
these mass-whitened coordinates. The mean-square \(L^2\) reconstruction
error separates into an effective-rank feature approximation term, a
deterministic time-consistency term, and a Bayesian estimation term. In the
high-information regime, the leading posterior contribution scales as
\(r\sigma^2/N_o\), where \(\sigma^2\) is the observation-noise variance and
\(N_o\) is the number of observations per analysis time. Thus the analysis
distinguishes the exact coefficient-space posterior from deterministic PDE
approximation errors, and gives a checkable uncertainty-quantification
guarantee for random-feature filtering of a representative parabolic PDE.
\end{abstract}



\pagestyle{myheadings}
\thispagestyle{plain}
\markboth{X.~LI, J.~LINGHU, AND Y.~WANG}{RANDOM-FEATURE KALMAN FILTERING FOR LINEAR PDE DA}

\section{Introduction}

Data assimilation for time-dependent PDEs starts from a basic mismatch:
the unknown state is an evolving field, whereas the data arrive as sparse,
noisy, and indirect measurements. Such problems arise in atmosphere and ocean
models, subsurface flow, seismic wave propagation, biomedical transport, and
engineering design. In Bayesian data assimilation, the target is not only a
state estimate, but the filtering distribution of the evolving state
conditioned on the observations
\cite{stuart2010inverse,dashti2017bayesian,evensen2009data,law2015data,sanzalonso2023inverse}.

For PDE models this creates a concrete UQ bottleneck. The state is
infinite-dimensional, observations arrive sequentially, sensors are sparse, and
matrix assembly may involve a large collocation or quadrature design. A useful
Bayesian filter should therefore keep a closed-form posterior covariance,
remain online, and have posterior dimension determined by the approximation
geometry rather than by the number \(N_o\) of observations at each analysis
time or the collocation count \(N_c\). Without such a coefficient state,
posterior uncertainty is tied either to ensemble sampling, batch optimization,
or a data-grid-sized kernel representation.

Existing approaches resolve different parts of this bottleneck. Ensemble
Kalman filters
\cite{evensen1994sequential,bishop2001adaptive,houtekamer2016review,iglesias2013ensemble,silva2025adaptive}
are sequential and scalable but
represent covariance through an ensemble. Variational methods such as 4DVar
optimize over an assimilation window
\cite{talagrand1987variational,courtier1994strategy} and recover covariance
through additional second-order calculations. Neural-Galerkin and shape-morphing parametrizations
\cite{aghili2024dynamical,hilliard2025sequential,ning2025filtered}
use nonlinear parameter manifolds, while Gaussian-process PDE solvers
\cite{raissi2018numerical,chen2021solving,bai2024gaussian}
keep Gaussian
structure in a kernel state tied to the collocation design. Reduced-basis and
POD filtering \cite{silva2025adaptive,iqbal2024parallel} reduce dimension
through a precomputed basis. For linear PDEs, these methods expose a narrow
but consequential UQ gap: conjugacy should be available if the right finite
coefficient state can be identified. The question is whether one can construct
a streaming coefficient model whose Bayesian update remains conjugate, whose
state dimension is decoupled from the sensor and collocation counts, and whose
posterior contraction can be quantified as observations accumulate.

The random feature method (RFM)
\cite{barron1993universal,rahimi2008random,bach2017equivalence,rudi2017generalization,chen2022bridging,chen2023timedependent,ming2025spectral}
provides such a coefficient space. Draw features
\(\{\varphi_j(\cdot;\omega_j)\}_{j=1}^M\) once and freeze them. The
ansatz
\[
u^M(x,t)=\Phi(x)^\top c(t),
\qquad
\Phi=(\varphi_1,\ldots,\varphi_M)^\top,
\]
is then linear in the coefficient vector \(c(t)\in\mathbb R^M\). Writing
\(c_n\) for the coefficient vector at analysis time \(t_n\), a linear PDE
with linear observations gives an affine state-space model
\[
c_{n+1}=F_nc_n+g_n+w_n,\qquad
y_n=H_nc_n+\varepsilon_n .
\]
Here \(y_n\) is the observation vector, \(F_n\) and \(H_n\) are the
coefficient propagator and observation matrix, \(g_n\) is the deterministic
forcing contribution, and \(w_n,\varepsilon_n\) denote model and observation
noise. With a Gaussian prior and Gaussian noise, the Kalman recursion gives
the posterior for the chosen coefficient model. The key point is that the
randomization is frozen before filtering: conditional on this draw, the model
is finite-dimensional and linear-Gaussian in the coefficient vector. The state
dimension is the raw feature count \(M\), or the effective rank \(r\le M\)
after mass whitening for non-orthogonal feature draws, and is separate from
\(N_o\) and \(N_c\).

This paper develops the above observation into a Bayesian random-feature
Kalman filter for linear PDE data assimilation. First, we formulate the
frozen-feature Galerkin system as a linear-Gaussian coefficient
state-space model and record the resulting exact Kalman posterior. Second,
we introduce mass-whitened effective-rank coordinates, which replace the
ill-conditioned raw random-feature coefficients by an orthonormal analysis
state of dimension \(r\). Third, for the heat equation with implicit-Euler
time stepping, we prove a posterior-contraction and PDE-consistency theorem
whose mean-square \(L^2\) reconstruction error separates into feature
approximation, time discretization, and Bayesian estimation. On the
theorem's high-probability event, the Bayesian-estimation term has leading
high-information scaling \(r\sigma^2/N_o\), where \(\sigma^2\) is the
observation-noise variance. The posterior statement is exact in coefficient
space; the comparison with the exact heat trajectory is obtained by adding
deterministic approximation and time-consistency terms.

The numerical experiments in Section~\ref{sec:numerics} follow the same
decomposition: they confirm the predicted heat-equation scalings, check
posterior calibration, compare with the ensemble-transform Kalman filter on a
streaming 2D heat benchmark, and test the coefficient-space filtering
mechanism on a non-self-adjoint advection-diffusion problem. Feature-family
specializations record inputs for Gaussian random Fourier features, shallow
ELU random features, and a Dirichlet-bubble construction for homogeneous
boundary conditions. The Supplementary Material records conditional extensions
to Caputo subdiffusion and to a linearized Darcy permeability inverse problem.

In the static one-shot limit, the construction reduces to Bayesian
PIELM \cite{liu2023bayesian}. Here the RFM coefficients evolve through a
PDE-induced state-space model, the Kalman recursion gives the
coefficient-space posterior, and the heat equation admits a quantitative
posterior-contraction/PDE-consistency theorem in effective-rank coordinates.
Table~\ref{tab:positioning} summarizes the relation to the closest prior work.

\begin{table}[t]
\centering
\footnotesize
\setlength{\tabcolsep}{4pt}
\begin{tabular}{p{0.20\linewidth}p{0.32\linewidth}p{0.36\linewidth}}
\toprule
Prior work & Update type & Present distinction \\
\midrule
BPIELM \cite{liu2023bayesian} & Closed-form, static & Streaming coefficient posterior \\
RAFDA \cite{gottwald2021supervised} & EnKF on learnt propagator & Exact coefficient posterior \\
NG filtering \cite{aghili2024dynamical,hilliard2025sequential,ning2025filtered} & Gauss-Newton on NN parameters & Closed-form coefficient posterior \\
GP/IEKS \cite{iqbal2024parallel} & Gauss-Newton MAP & Effective-rank coefficient state \\
RB-EnKF \cite{silva2025adaptive} & EnKF on POD modes & Random-feature coefficient posterior \\
Numerical GP \cite{raissi2018numerical,chen2021solving,bai2024gaussian} & GP-PDE with kernel state & Dimension decoupled from \(N_c,N_o\) \\
iPINNER \cite{lu2025ipinner} & PINN + EnKF, iterated & Linear-Gaussian coefficient recursion \\
\bottomrule
\end{tabular}
\caption{Positioning against the closest prior work in PDE-informed
Bayesian data assimilation. The proposed framework gives a closed-form
Kalman update and exact coefficient-space posterior on effective-rank
random-feature coefficients, with state dimension independent of \(N_c\) and
\(N_o\).}
\label{tab:positioning}
\end{table}

The rest of the paper is organized as follows.
Section~\ref{sec:framework} fixes the PDE class, the RFM Galerkin
discretization, and the Kalman recursion.
Section~\ref{sec:theory} develops the mass-whitened effective-rank state space,
proves the linear-heat posterior contraction theorem with
traceable constants, and records the family-specific corollaries.
Section~\ref{sec:numerics} reports the numerical experiments, and
Section~\ref{sec:discussion} closes with operating regimes and
implications.

\section{The Random-Feature Kalman Filtering Framework}
\label{sec:framework}

This section isolates the finite-dimensional object for which the Bayesian
statement is exact. We first recall Bayesian filtering for data assimilation
(\S\ref{sec:bayesian-filtering}), then build an RFM Galerkin coefficient model
for a linear PDE (\S\ref{sec:rfm-trial-space}), and finally record the Kalman
recursion for the resulting linear-Gaussian state space
(\S\ref{sec:exact-kalman}).

\subsection{Bayesian filtering for PDE data assimilation}
\label{sec:bayesian-filtering}

Data assimilation infers a time-dependent state of a specified
dynamical model from streams of noisy, sparse observations. In
discrete time and finite dimension, the model is a state-space
pair
\begin{equation}\label{eq:state-space}
x_{n+1}=\mathcal F_n(x_n)+w_n,
\qquad
y_n=\mathcal H_n(x_n)+\varepsilon_n,
\end{equation}
with state \(x_n\in\mathbb R^{d_x}\), observation
\(y_n\in\mathbb R^{N_o}\), process noise
\(w_n\sim\mathcal N(0,Q_n)\), observation noise
\(\varepsilon_n\sim\mathcal N(0,R_n)\), and a Gaussian prior
\(x_0\sim\mathcal N(m_0,P_0)\). The filtering distribution is
\(\pi_n(\cdot)=\mathbb P(x_n\mid y_{0:n})\). For linear
\(\mathcal F_n\) and \(\mathcal H_n\), it remains
Gaussian and is computed by the Kalman filter
\cite{kalman1960new,anderson1979optimal,sanzalonso2023inverse}.

PDE data assimilation specializes \eqref{eq:state-space} to the
infinite-dimensional setting where the state is a function
\(u(\cdot,t_n)\) on a spatial domain \(\Omega\) governed by a
partial differential operator. For the construction below, the relevant
distinction is between ensemble methods
\cite{silva2025adaptive}, which approximate the filtering
distribution by a Monte Carlo sample with sampling error and
covariance-inflation choices, and discretize-then-filter
approaches, which Galerkin-project the PDE onto a
finite-dimensional trial space and apply a Kalman filter on the resulting
coefficient state \cite{thomee2006galerkin}. For a fixed linear coefficient model with Gaussian
noise, conjugacy gives the Bayesian update in closed form. The trial space
therefore determines both the numerical approximation and the dimension of
the posterior state.

\subsection{The random feature method as the trial space}
\label{sec:rfm-trial-space}

The random feature method
\cite{rahimi2008random,chen2022bridging,chen2023timedependent,ming2025spectral}
parametrizes a function on \(\Omega\) by a fixed-size linear
combination
\begin{equation}\label{eq:rfm-ansatz}
u^M(x,t)
=\sum_{j=1}^M c_j(t)\,\varphi_j(x;\omega_j)
=\Phi(x)^\top c(t),
\qquad
c(t)\in\mathbb R^M,
\end{equation}
where the features \(\{\varphi_j(\cdot;\omega_j)\}_{j=1}^M\) are
sampled once from a distribution
\(\omega_j\stackrel{\rm iid}{\sim}\pi\) and \emph{frozen} for the
remainder of the computation, and
\(\Phi=(\varphi_1,\ldots,\varphi_M)^\top\) is the
column-stacked feature vector. Examples below include Gaussian random
Fourier features and shallow random networks, with deterministic
eigenbases used as reference trial spaces. Once the draw is taken,
\(V_M=\operatorname{span}\{\varphi_j\}_{j=1}^M\) is fixed and
finite-dimensional. When a conforming Galerkin stiffness matrix is
formed, we assume \(V_M\subset D(a)\), where \(a\) is the closed
bilinear form associated with \(-\mathcal L\). Nonconforming assembled
stiffness matrices are treated as part of the finite-dimensional
coefficient model, with stability supplied as an explicit input. When
point observations are used, the features are taken with pointwise
representatives on the observation set.

\subsubsection{PDE class and Galerkin discretization}
\label{ssec:framework-pde-galerkin}

Fix a working domain \(\Omega\subset\mathbb R^d\) (bounded with
\(C^{1,1}\) boundary), a final time \(T>0\), and a uniform
analysis grid \(t_n=n\Delta t\) for
\(0\le n\le N=T/\Delta t\). Throughout, \(N_o\) denotes the
number of observations per analysis time, \(N_c\) denotes a
collocation or quadrature count when such points are used to assemble
matrices, and ensemble sizes in the ETKF benchmarks are denoted by
\(N_e\). The prototype problem is the linear
parabolic equation
\begin{equation}\label{eq:pde-class}
\partial_t u(x,t)=\mathcal Lu(x,t)+f(x,t),
\qquad u|_{\partial\Omega}=0,
\qquad u(\cdot,0)=u_0\in L^2(\Omega),
\end{equation}
with linear, uniformly elliptic, dissipative generator
\(\mathcal L\) on \(L^2(\Omega)\); the running example is
\(\mathcal L=\nu\Delta\) with fixed diffusivity \(\nu>0\) (heat
operator). The diffusivity is part of the bilinear form and is therefore
absorbed into the stiffness matrix. When a positive operator is needed,
we write \(A=-\mathcal L\succeq0\), with effective-rank counterpart \(A_\tau\)
in Section~\ref{ssec:contraction-theorem}. With
\(f\in L^\infty(0,T;L^2(\Omega))\), the Galerkin mass and stiffness
matrices are
\(\mathbb M_{kl}=\int_\Omega\varphi_k\varphi_l\,dx\) and
\(\mathbb K_{kl}=a(\varphi_l,\varphi_k)\). Here
\(a(u,v)=-(\mathcal L u,v)_{L^2}\) when the strong form is valid,
and \(a(u,v)=\nu\int_\Omega\nabla u\cdot\nabla v\,dx\) for
conforming homogeneous-Dirichlet heat. The coefficient equation is
\(\mathbb M\dot c(t)=-\mathbb K c(t)+\tilde f(t)\), with
\(\tilde f_k(t)=\int_\Omega\varphi_k f(\cdot,t)\,dx\). For
self-adjoint dissipative \(\mathcal L\), \(\mathbb K\) is
symmetric positive semi-definite on conforming spaces; otherwise
the assembled coefficient model is treated through the explicit
finite-horizon stability input used in Section~\ref{sec:theory}. Section
\ref{sec:theory} normalizes the \(L^2\) inner product by \(|\Omega|^{-1}\);
this rescales the mass, stiffness, and load matrices consistently and only
changes constants.

\subsubsection{Time integration and observation}
\label{ssec:framework-time-observation}

An implicit-explicit Runge-Kutta scheme of order \(p\) applied
to the semi-discrete ODE gives, in the chosen coefficient coordinate,
\(c(t_{n+1})=F_n c(t_n)+g_n+\rho_n\), with
\(\|\rho_n\|_2\le C_{\rm time}(\Delta t)^{p+1}\). The filtering state
\(c_n\) follows the numerical recursion obtained by dropping
\(\rho_n\); Section~\ref{sec:theory} accounts for this term as
time-discretization error, while the state-space model may include
explicit process noise \(w_n\). For implicit Euler,
\(F=(\mathbb M+\Delta t\,\mathbb K)^{-1}\mathbb M\), when
\(\mathbb M+\Delta t\,\mathbb K\) is nonsingular, is
mass-norm contractive for \(\mathbb K\succeq0\). The effective-rank
mass-whitening construction removes near-null mass directions and supplies
the Euclidean bound used in the contraction analysis.
Point observations at sensor
locations \(\{X_i\}_{i=1}^{N_o}\subset\Omega\) and analysis
times \(t_n\) give the linear measurement equation
\(y_n=H_n c_n+\varepsilon_n\) with \((H_n)_{ij}=\varphi_j(X_i)\)
and \(\varepsilon_n\stackrel{\rm iid}{\sim}
\mathcal N(0,\sigma^2I_{N_o})\). Allowing \(H_n\) to vary with
\(n\) accommodates random observation locations.

\subsection{Kalman recursion for the coefficient model}
\label{sec:exact-kalman}

Condition on the frozen feature draw, quadrature or collocation rule, and
observation design. The coefficient state-space model is
\begin{equation}\label{eq:coefficient-state-space}
c_{n+1}=F_nc_n+g_n+w_n,\qquad
y_n=H_nc_n+\varepsilon_n ,
\end{equation}
with \(F_n\in\mathbb R^{M\times M}\), \(g_n\in\mathbb R^M\) for
\(0\le n<N\), and \(H_n\in\mathbb R^{N_o\times M}\) for
\(0\le n\le N\). The matrices are deterministic after conditioning on the
feature draw and the observation design; randomness enters through the prior,
process increments, and observation errors.

After this conditioning, the PDE and feature construction have supplied
deterministic matrices. The remaining inputs needed for conjugacy are Gaussian
prior/noise laws, independence, and \(R_n\succ0\); these are probabilistic
inputs, separate from the approximation assumptions used later.

\begin{assumption}[Conjugate Gaussian inputs]\label{ass:lg-rfm}
In \eqref{eq:coefficient-state-space},
\[
c_0\sim\mathcal N(m_0,P_0),\qquad
w_n\sim\mathcal N(0,Q_n),\qquad
\varepsilon_n\sim\mathcal N(0,R_n),
\]
for \(w_n\) with \(0\le n<N\) and \(\varepsilon_n\) with \(0\le n\le N\),
where \(m_0\in\mathbb R^M\), \(P_0,Q_n\in\mathbb R^{M\times M}\),
\(P_0\succeq0\), \(Q_n\succeq0\), \(R_n\in\mathbb R^{N_o\times N_o}\), and
\(R_n\succ0\). The variables \(c_0\), \(\{w_n\}_{n=0}^{N-1}\), and
\(\{\varepsilon_n\}_{n=0}^N\) are mutually independent.
\end{assumption}

Affine Gaussian propagation and Gaussian conditioning then give the exact
coefficient posterior.

\begin{maintheorem}[Exact Kalman recursion for linear RFM assimilation]
\label{thm:linear-gaussian-rfm}
For the coefficient model \eqref{eq:coefficient-state-space} under
Assumption~\ref{ass:lg-rfm}, every filtering distribution
\(c_n\mid y_{0:n}\) is Gaussian for \(0\le n\le N\). Set
\(m_{0|-1}=m_0\) and \(P_{0|-1}=P_0\). For \(n=0\), this is the
prior forecast; for \(1\le n\le N\), suppose
\(c_n\mid y_{0:n-1}\sim
\mathcal N(m_{n|n-1},P_{n|n-1})\). The analysis update is
\begin{equation}\label{eq:kalman-analysis}
\begin{aligned}
S_n&=H_nP_{n|n-1}H_n^\top+R_n,
\qquad
K_n=P_{n|n-1}H_n^\top S_n^{-1},\\
m_n&=m_{n|n-1}+K_n\bigl(y_n-H_nm_{n|n-1}\bigr),\\
P_n&=(I-K_nH_n)P_{n|n-1}(I-K_nH_n)^\top
+K_nR_nK_n^\top,
\end{aligned}
\end{equation}
for \(0\le n\le N\), where the Joseph form
\cite{anderson1979optimal} preserves symmetry and positive
semi-definiteness under finite-precision arithmetic. For
\(0\le n\le N-1\), the forecast is
\begin{equation}\label{eq:kalman-forecast}
m_{n+1|n}=F_nm_n+g_n,
\qquad
P_{n+1|n}=F_nP_nF_n^\top+Q_n .
\end{equation}
Thus, within the fixed coefficient model, the posterior update is
closed form; the remaining approximations are the RFM basis, time
discretization, effective-rank reduction, and noise model.
\end{maintheorem}

\begin{proof}
This is the classical linear-Gaussian Kalman recursion
\cite{anderson1979optimal}. The prior gives the first forecast,
Gaussian conditioning gives the analysis update, and the affine
transition \(c_{n+1}=F_nc_n+g_n+w_n\) gives the next forecast.
Induction over \(n\) completes the recursion.
\end{proof}

In Section~\ref{sec:theory}, we write \(m_n^f=m_{n|n-1}\) and
\(P_n^f=P_{n|n-1}\) for forecast quantities, and
\(m_n^a=m_n\), \(P_n^a=P_n\) for analysis quantities.
Theorem~\ref{thm:linear-gaussian-rfm} is the coefficient-space reduction
used below. Section~\ref{sec:theory} specializes it to deterministic
effective-rank heat dynamics, constructs a stable mass-whitened coordinate,
controls empirical Gram matrices, and bounds posterior contraction in
\(M\), \(\Delta t\), \(N_o\), and \(\sigma\).

\section{Posterior Contraction and Consistency}
\label{sec:theory}

Theorem~\ref{thm:linear-gaussian-rfm} gives a closed-form Kalman posterior
once the PDE has been restricted to a finite random-feature state space. The
remaining question is quantitative: for the linear heat equation, how do
the posterior covariance and the reconstructed \(L^2\) field error scale
with the effective rank, the observation count, and the time
step? Theorem~\ref{thm:linear-heat-retained-contraction} answers this
question on a joint high-probability event. Its bound separates three
terms: effective-rank feature approximation, deterministic time
consistency, and Bayesian estimation.

The contraction theorem is stated for homogeneous heat, \(f=0\), so that
the analysis can focus on posterior contraction and deterministic
feature/time consistency. A known forcing term adds the corresponding
deterministic quadrature and time-integration consistency contribution to
\(g_n\).

The proof is organized as a set of checkable certificates rather than as a
new modeling layer. Section~\ref{ssec:retained-state-space} constructs the
mass-whitened effective-rank coefficient space in which the Kalman posterior
is analyzed. Section~\ref{ssec:concentration-events} records three
certificates, approximation, leverage, and empirical-Gram concentration, that
connect the random feature draw and random sensors to deterministic bounds.
Section~\ref{ssec:contraction-theorem} adds the deterministic heat-consistency
certificate and combines it with the Kalman covariance identity. Finally,
Section~\ref{ssec:feature-corollaries} identifies feature-family inputs for
Gaussian random Fourier, ELU, and Dirichlet-bubble features. Conditional
extensions to Caputo subdiffusion and linearized Darcy inversion are given in
the Supplementary Material; they require the same effective-state and Gram
arguments together with problem-specific stability or linearization
assumptions.

\subsection{Mass-whitened effective-rank coordinates}
\label{ssec:retained-state-space}

Let \(\Phi=(\varphi_1,\ldots,\varphi_M)^\top\) denote the raw feature
vector. For non-orthogonal random features, the Galerkin mass matrix may
have very small tail eigenvalues, and estimates stated in the raw
coefficient vector \(c\in\mathbb R^M\) inherit this conditioning. The
following construction replaces the raw coefficients by a mass-whitened
effective coordinate in which the mass matrix is the identity. We call \(r\)
the effective rank and \(z\in\mathbb R^r\) the mass-whitened effective
coordinate. The posterior dimension is \(r\), not the raw feature count \(M\).
Only the nondegeneracy condition \(r\ge1\) is assumed; the rest is the
deterministic coordinate construction used by the theorem.

\begin{construction}[Mass-whitened effective-rank coordinates]
\label{ass:mass-whitened-rfm}
Let \(\{\varphi_j\}_{j=1}^M\) be a fixed raw random-feature draw on
\(\Omega\), and let
\[
d\mu_\Omega(x)=|\Omega|^{-1}\,dx
\]
be normalized volume measure. All \(L^2\) inner products in this section
are taken with respect to \(\mu_\Omega\); when \(|\Omega|=1\), this is the
usual Lebesgue \(L^2\) inner product. Thus
\(\|v\|_{L^2(\Omega)}^2=\int_\Omega |v(x)|^2\,d\mu_\Omega(x)\) throughout
this section. For the heat-equation results, fix the diffusivity \(\nu>0\)
and define
\[
\mathbb M_{ij}=\int_\Omega \varphi_i(x)\varphi_j(x)\,d\mu_\Omega(x),
\qquad
\mathbb K_{ij}=\nu\int_\Omega
\nabla\varphi_i(x)\cdot\nabla\varphi_j(x)\,d\mu_\Omega(x).
\]
Assume \(\mathbb M\succeq0\) and let
\(\mathbb M=V\Lambda V^\top\), where
\(\Lambda=\operatorname{diag}(\lambda_1,\ldots,\lambda_M)\) with
\(\lambda_1\ge\cdots\ge\lambda_M\). For a tolerance \(\tau\ge0\), define
\[
I_\tau=\{j:\lambda_j>\tau\lambda_1\},\qquad r=|I_\tau|,
\]
and assume \(r\ge1\). Let \(V_\tau\) contain the retained eigenvectors and
\(\Lambda_\tau\) the retained eigenvalues. The coefficient transform is
\[
T_\tau=V_\tau\Lambda_\tau^{-1/2}\in\mathbb R^{M\times r}.
\]
The effective feature vector is
\[
\Psi(x)=(\psi_1(x),\ldots,\psi_r(x))^\top=T_\tau^\top\Phi(x).
\]
All coefficient-space dynamics and observations are then written in the
mass-whitened effective coordinate \(z\in\mathbb R^r\), with transformed matrices
\[
\widetilde{\mathbb M}=T_\tau^\top\mathbb M T_\tau,\qquad
\widetilde{\mathbb K}=T_\tau^\top\mathbb K T_\tau,\qquad
\widetilde H_n=H_nT_\tau.
\]
Below, \(H_n\) denotes the effective-rank observation matrix unless the raw
coefficient matrix is explicitly mentioned. The discarded eigendirections
enter through the deterministic approximation term in the contraction theorem.
\end{construction}

\begin{remark}[Continuum and quadrature whitening]
\label{rem:quadrature-whitening}
Construction~\ref{ass:mass-whitened-rfm} is a continuum statement. If the
features are whitened by a fixed quadrature rule, the same argument applies
with an additional deterministic consistency term measuring the departure from
continuum orthonormality, for example
\[
\eta_{\rm quad}
=\left\|\int_\Omega\Psi(x)\Psi(x)^\top\,d\mu_\Omega(x)-I_r\right\|_2 .
\]
The main theorem is stated for \(\eta_{\rm quad}=0\).
\end{remark}

\begin{lemma}[Mass identity after whitening]\label{lem:mass-whitening}
Under Construction~\ref{ass:mass-whitened-rfm},
\[
\widetilde{\mathbb M}=I_r.
\]
Consequently, the implicit-Euler heat propagator in the effective coordinates
is
\[
F=(I_r+\Delta t\,\widetilde{\mathbb K})^{-1}.
\]
\end{lemma}

\begin{proof}
By construction,
\[
T_\tau^\top\mathbb M T_\tau
=\Lambda_\tau^{-1/2}V_\tau^\top
V_\tau\Lambda_\tau V_\tau^\top
V_\tau\Lambda_\tau^{-1/2}=I_r,
\]
using the orthonormality of the retained eigenvectors. Substituting
\(\widetilde{\mathbb M}=I_r\) into the mass-form implicit-Euler update
\((\widetilde{\mathbb M}+\Delta t\,\widetilde{\mathbb K})z_{n+1}
=\widetilde{\mathbb M}z_n\) gives the stated propagator.
\end{proof}

Thus Theorem~\ref{thm:linear-gaussian-rfm} applies with \(M\) replaced by
\(r\). Inside the effective coefficient space, the Kalman recursion is closed form, and for
implicit-Euler heat the propagator in Lemma~\ref{lem:mass-whitening} is
contractive whenever \(\widetilde{\mathbb K}\succeq0\). The remaining task is
to quantify how the coefficient-space posterior covariance and field reconstruction
error depend on \(r\), \(N_o\), and \(\Delta t\).

\subsection{Approximation and observation certificates}
\label{ssec:concentration-events}

The contraction theorem is conditioned on three proof certificates. They are
not additional dynamics assumptions: they record where approximation, feature
geometry, and sensor randomness enter the estimate. The approximation
certificate places the exact PDE trajectory near the effective-rank feature
space. The leverage certificate bounds the pointwise size of the effective
basis. The Gram certificate controls the empirical observation operator
generated by random sensor locations. We keep them separate because they
involve different sources of randomness: the frozen feature draw and the
observation locations.

\subsubsection{Approximation}
It places the exact heat trajectory near the frozen effective-rank space and
provides the deterministic bridge to an effective coefficient reference.

\begin{assumption}[Effective-rank approximation class]
\label{ass:retained-approximation}
Fix a final time \(T=N\Delta t\), a raw feature family, and a
mass-retention tolerance \(\tau\). For the exact PDE states
\(u^\star(\cdot,t_n)\), \(0\le n\le N\), assume that with probability at
least \(1-\delta_{\rm app}\) over the frozen feature draw there is a
rank-truncated mass-whitened feature space satisfying
Construction~\ref{ass:mass-whitened-rfm} and coefficients
\(z_n^\dagger\in\mathbb R^r\) such that
\[
\sup_{0\le n\le N}
\left\|u^\star(\cdot,t_n)-\Psi^\top z_n^\dagger\right\|_{L^2(\Omega)}
\le \varepsilon_{M,\tau}(u^\star).
\]
\end{assumption}

For Gaussian random Fourier features this event follows from standard
Monte Carlo approximation in the associated kernel/Barron-type integral
class \cite{rahimi2008random,bach2017equivalence}, up to the discarded
mass-eigendirection tail. For ELU random features it is the analogous
single-hidden-layer RFM approximation hypothesis
\cite{chen2022bridging,chen2023timedependent,ming2025spectral}. The
feature-family statements at the end of this section provide the
corresponding constants and failure probabilities.

\subsubsection{Leverage}
The leverage certificate bounds the pointwise effective-rank leverage. This is
the only feature-dependent quantity entering the empirical Gram concentration
estimate.

\begin{assumption}[Effective-rank leverage event]
\label{ass:retained-leverage}
For a mass-whitened effective-rank feature draw define
\[
L_\tau=\operatorname*{ess\,sup}_{x\in\Omega}\|\Psi(x)\|_2^2,
\]
where the essential supremum is with respect to \(\mu_\Omega\). Since
\[
\mathbb E_{X\sim\mu_\Omega}\|\Psi(X)\|_2^2
=\operatorname{tr}\int_\Omega\Psi(x)\Psi(x)^\top\,d\mu_\Omega(x)=r,
\]
one always has \(L_\tau\ge r\), and any valid certificate satisfies
\(L_\star\ge r\). For a deterministic level \(L_\star\ge r\), let
\(\mathcal E_{\rm lev}(L_\star)=\{L_\tau\le L_\star\}\) denote the leverage
event. The contraction theorem is stated on
\(\mathcal E_{\rm lev}(L_\star)\).
\end{assumption}

The theorem only needs the event \(L_\tau\le L_\star\). For readers who want a
checkable finite certificate, the grid-certificate lemma in the Supplementary
Material shows how to certify \(L_\star\) from a
finite cover of \(\Omega\), raw feature envelopes, and gradient envelopes. In
computations this cover can be chosen from the quadrature grid used to assemble
\(\mathbb M\); the feature-family envelopes are summarized in
Section~\ref{ssec:feature-corollaries}.

\subsubsection{Empirical Gram}
The Gram certificate is the matrix Chernoff concentration
\cite{tropp2012user} of
\(N_o^{-1}H^\top H\) around the continuum mass-whitened identity. The
sample size depends on the effective-rank leverage envelope \(L_\star\).

\begin{lemma}[Conditional empirical Gram concentration]
\label{lem:retained-gram-concentration}
Let a feature draw satisfy Construction~\ref{ass:mass-whitened-rfm} and the
leverage event \(\mathcal E_{\rm lev}(L_\star)\) in
Assumption~\ref{ass:retained-leverage}. Draw observation locations
\(X_1,\ldots,X_{N_o}\) independently from \(\mu_\Omega\), independently of the
feature draw, and define \(H_{ia}=\psi_a(X_i)\). Then, for \(0<\eta<1\),
\[
\Pr\left[
\lambda_{\min}\left(N_o^{-1}H^\top H\right)\le 1-\eta
\;\middle|\;\Psi
\right]
\le
r\left(\frac{e^{-\eta}}{(1-\eta)^{1-\eta}}\right)^{N_o/L_\star},
\]
and
\[
\Pr\left[
\lambda_{\max}\left(N_o^{-1}H^\top H\right)\ge 1+\eta
\;\middle|\;\Psi
\right]
\le
r\left(\frac{e^\eta}{(1+\eta)^{1+\eta}}\right)^{N_o/L_\star}.
\]
Consequently, the two displayed Chernoff factors simplify to a bound of the
form \(r\exp(-cN_o\eta^2/L_\star)\) with a universal \(c>0\), so that
there is a universal constant \(C\) such that the deviation event
\[
\left\|N_o^{-1}H^\top H-I_r\right\|_2\le\eta
\]
holds with conditional probability at least \(1-\delta_{\rm gram}\)
whenever
\[
N_o\ge C L_\star\eta^{-2}\log\frac{2r}{\delta_{\rm gram}} .
\]
\end{lemma}

This is the observation-information certificate. Mass whitening makes the
continuum covariance of \(\Psi(X)\) equal to \(I_r\); the leverage bound is the
only feature-dependent quantity needed to apply matrix Chernoff.

A union bound over the \(N+1\) analysis times extends this to the
time-uniform event needed for streaming assimilation. The cost is a
logarithmic factor in \(N\); for re-used observation locations the
single-time lemma is already uniform in time.

\begin{lemma}[Time-uniform effective-rank Gram concentration]
\label{lem:time-uniform-retained-gram}
Let a feature draw satisfy Construction~\ref{ass:mass-whitened-rfm} and the
leverage event \(\mathcal E_{\rm lev}(L_\star)\). For each analysis time
\(t_n\), \(0\le n\le N\), draw observation locations
\(X_{n,1},\ldots,X_{n,N_o}\) independently from \(\mu_\Omega\), independently
across \(n\) and independently of the feature draw, and define
\((H_n)_{ia}=\psi_a(X_{n,i})\). Then there is a universal constant
\(C>0\) such that, for \(0<\eta<1\), whenever
\(N_o\ge C L_\star\eta^{-2}\log(2r(N+1)/\delta_{\rm gram})\) the
deviation event
\(\max_{0\le n\le N}\|N_o^{-1}H_n^\top H_n-I_r\|_2\le\eta\) holds
with conditional probability at least \(1-\delta_{\rm gram}\), given
the mass-whitened feature draw. If the same observation locations are
reused at every analysis time, the same conclusion holds with the
logarithmic factor \(\log(2r/\delta_{\rm gram})\), because there is
only one empirical Gram matrix to control.
\end{lemma}

Thus the single-analysis information certificate becomes a streaming
certificate by a union bound over analysis times. Appendix~\ref{app:aux-proofs}
gives the time-uniform argument; the Supplementary Material gives the
single-time Chernoff proof and the grid-leverage certificate.

\subsection{Risk decomposition theorem}
\label{ssec:contraction-theorem}

We next relate the effective-rank Kalman model to the exact heat trajectory.
Lemma~\ref{lem:deterministic-retained-dynamics-bias} accumulates the named
dynamics input \(\mathcal E_{\rm dyn}(C_b,p,S_T)\), and
Lemmas~\ref{lem:implicit-euler-heat-consistency}--\ref{lem:ritz-heat-consistency}
verify that input for heat. The stochastic part is separate: the empirical
Gram event controls the Kalman covariance, and
Lemma~\ref{lem:deterministic-truth-kalman-mean-error} converts that covariance
into a mean-square bound for data from a deterministic effective-rank
reference trajectory.

\subsubsection{Dynamics Bias}
The first ingredient in the risk decomposition is deterministic. It measures
how well the forecast matrices propagate the chosen effective-rank reference
trajectory.

\begin{assumption}[Effective-rank dynamics consistency input]
\label{ass:retained-dynamics-consistency}
Fix \(C_b>0\), \(p>0\), and \(S_T\ge1\). Let
\(F_n\in\mathbb R^{r\times r}\), \(0\le n<N\), be the effective-rank forecast
matrices used by the Kalman model, and let
\(\{z_n^\dagger\}_{n=0}^N\subset\mathbb R^r\) be the reference
coefficient trajectory. Define
\[
b_{n+1}=z_{n+1}^\dagger-F_nz_n^\dagger,\qquad 0\le n<N,
\]
and, for \(0\le m\le n\le N\),
\[
\mathcal F_{n:m}=
\begin{cases}
F_{n-1}F_{n-2}\cdots F_m, & m<n,\\
I_r, & m=n.
\end{cases}
\]
The input \(\mathcal E_{\rm dyn}(C_b,p,S_T)\) holds if
\[
\|\mathcal F_{n:m}\|_2\le S_T,\qquad 0\le m\le n\le N,
\]
and
\[
\|b_{n+1}\|_2
\le
C_b\Delta t\left(\varepsilon_{M,\tau}(u^\star)+(\Delta t)^p\right),
\qquad 0\le n<N .
\]
\end{assumption}

\begin{lemma}[Deterministic effective-rank dynamics bias]
\label{lem:deterministic-retained-dynamics-bias}
Let Construction~\ref{ass:mass-whitened-rfm} and the dynamics input
\(\mathcal E_{\rm dyn}(C_b,p,S_T)\) of
Assumption~\ref{ass:retained-dynamics-consistency} hold.
Let the model-only coefficient trajectory satisfy
\[
\bar z_0=z_0^\dagger,\qquad \bar z_{n+1}=F_n\bar z_n.
\]
Then, for every \(0\le n\le N\),
\[
\left\|z_n^\dagger-\bar z_n\right\|_2
\le
S_T\sum_{k=0}^{n-1}\|b_{k+1}\|_2.
\]
Consequently, by the whitening identity,
\[
\left\|\Psi^\top(z_n^\dagger-\bar z_n)\right\|_{L^2(\Omega)}
\le
S_T\sum_{k=0}^{n-1}\|b_{k+1}\|_2.
\]
In particular,
\[
\sup_{0\le n\le N}
\left\|\Psi^\top(z_n^\dagger-\bar z_n)\right\|_{L^2(\Omega)}
\le
S_T C_bT
\left(\varepsilon_{M,\tau}(u^\star)+(\Delta t)^p\right).
\]
\end{lemma}

\begin{proof}
Set \(e_n=z_n^\dagger-\bar z_n\). Since \(e_0=0\),
\[
e_{n+1}
=z_{n+1}^\dagger-F_n\bar z_n
=F_ne_n+b_{n+1}.
\]
Iterating this recursion gives
\[
e_n=\sum_{k=0}^{n-1}\mathcal F_{n:k+1}b_{k+1},
\]
with the convention \(\mathcal F_{n:n}=I_r\) for the last term. The stability
bound and the triangle inequality imply the coefficient estimate. The
\(L^2\)-field estimate follows from Lemma~\ref{lem:mass-whitening}, which
gives
\(\|\Psi^\top a\|_{L^2(\Omega)}=\|a\|_2\) for effective coordinates. The
displayed consistency consequence follows from
\(\mathcal E_{\rm dyn}(C_b,p,S_T)\) by summing the residual bound and using
\(n\Delta t\le T\).
\end{proof}

\subsubsection{Heat Consistency}
The dynamics input above is abstract on purpose: it separates the main Kalman
argument from the PDE-specific consistency estimate. For heat, there are two
routes. The \(L^2\)-projection route uses a graph-domain residual and is
appropriate when the generator action on the projected trajectory is
well-defined. The conforming route uses the Galerkin heat solution and a Ritz
approximation quantity; this is the route used later for Dirichlet-bubble
features.

For an \(L^2\)-projection reference,
\(\mathcal E_{\rm dyn}(C_b,p,S_T)\) can be verified by controlling the heat
generator in the graph norm. For a general nonconforming random-feature
space, the operator \(\mathcal L\) need not act on \(P_\tau u^\star\) as an
\(L^2\) function. The following certificate is therefore conditional on the
displayed graph-domain residual being well defined and controlled. This
condition is automatic for commuting spectral truncations and is replaced by
the conforming Ritz argument of Lemma~\ref{lem:ritz-heat-consistency} for
boundary-compatible Galerkin spaces.

\begin{lemma}[Graph-residual consistency certificate for \(L^2\)-projection references]
\label{lem:implicit-euler-heat-consistency}
Let Construction~\ref{ass:mass-whitened-rfm} hold and let
\(F=(I_r+\Delta t\,\widetilde{\mathbb K})^{-1}\) be the implicit-Euler
effective-rank heat propagator with \(\widetilde{\mathbb K}\succeq0\). Let
\(u^\star\) be a strong solution of the linear heat equation on
\(\Omega\) satisfying \(\partial_t u^\star=\mathcal Lu^\star\), and
choose reference coefficients
\[
z_n^\dagger=\arg\min_{z\in\mathbb R^r}
\|u^\star(\cdot,t_n)-\Psi^\top z\|_{L^2(\Omega)},
\qquad 0\le n\le N,
\]
i.e. the orthogonal projection onto the effective-rank span; denote this
projector by \(P_\tau\). Assume the
approximation event in Assumption~\ref{ass:retained-approximation}
holds with constant \(\varepsilon_{M,\tau}(u^\star)\). Because the heat
generator is unbounded, assume in addition that the displayed
graph residual is well defined in \(L^2(\Omega)\) and satisfies the
graph-domain generator-consistency condition
\[
\sup_{0\le n\le N}
\left\|P_\tau\mathcal L
\left(P_\tau u^\star(\cdot,t_n)-u^\star(\cdot,t_n)\right)
\right\|_{L^2(\Omega)}
\le C_{\mathcal L}\varepsilon_{M,\tau}(u^\star).
\]
For commuting spectral truncations this condition is automatic; for
non-orthogonal random features it is a graph-domain input that
supplements the \(L^2\) approximation event. Finally, assume
\(t\mapsto u^\star(\cdot,t)\in C^2([0,T];L^2(\Omega))\) and define
\[
M_2=\sup_{0\le t\le T}
\|\partial_{tt}u^\star(\cdot,t)\|_{L^2(\Omega)}<\infty.
\]
Then the residual
\(b_{n+1}=z_{n+1}^\dagger-Fz_n^\dagger\) defined in
Assumption~\ref{ass:retained-dynamics-consistency} satisfies the
two-term estimate
\[
\|b_{n+1}\|_2
\le
\Delta t\,C_{\mathcal L}\,\varepsilon_{M,\tau}(u^\star)
+\tfrac12 M_2(\Delta t)^2 ,
\qquad 0\le n<N,
\]
where \(C_{\mathcal L}\) is the constant in the graph-domain
generator-consistency condition. Hence
\(\mathcal E_{\rm dyn}(C_{\mathcal L}+\tfrac12M_2,1,1)\) holds, and
Lemma~\ref{lem:deterministic-retained-dynamics-bias} gives the sharper
accumulated dynamics-bias bound
\[
\sup_{0\le n\le N}
\left\|\Psi^\top(z_n^\dagger-\bar z_n)\right\|_{L^2(\Omega)}
\le
T\,C_{\mathcal L}\,\varepsilon_{M,\tau}(u^\star)
+\tfrac12 M_2 T\,\Delta t ,
\]
with \(S_T=1\) coming from \(\|F\|_2\le 1\). The bound is first order in
\(\Delta t\) and linear in the effective-rank approximation error.
\end{lemma}

Lemma~\ref{lem:implicit-euler-heat-consistency} is the nonconforming route: it
shows exactly what extra graph-domain information is needed when the reference
trajectory is the \(L^2\)-projection onto the effective feature space. This is
useful for spectral truncations and for nonconforming feature families when
the displayed graph residual can be certified. Boundary-compatible random
features allow a cleaner route. In that case the reference trajectory is the
Galerkin heat solution itself, the weak form handles the unbounded generator,
and a parabolic Ritz quantity replaces the graph residual. The proof of
Lemma~\ref{lem:implicit-euler-heat-consistency} is sketched in
Appendix~\ref{app:aux-proofs}; the Supplementary Material gives the remaining
technical details.

\begin{lemma}[Conforming Galerkin consistency via Ritz approximation]
\label{lem:ritz-heat-consistency}
Let \(A=-\mathcal L\) be the homogeneous-Dirichlet heat operator with
fixed diffusivity \(\nu>0\) and bilinear form
\[
a(u,v)=\nu\int_\Omega \nabla u(x)\cdot\nabla v(x)\,d\mu_\Omega(x),
\qquad u,v\in H^1_0(\Omega).
\]
Let \(V_\tau=\operatorname{span}\{\psi_1,\ldots,\psi_r\}\subset H^1_0(\Omega)\)
be a conforming effective-rank space satisfying
Construction~\ref{ass:mass-whitened-rfm}, and let
\(R_\tau:H^1_0(\Omega)\to V_\tau\) be the Ritz projection
\[
a(R_\tau u-u,v)=0,\qquad v\in V_\tau .
\]
Let \(u_\tau^{\rm G}(t)=\Psi^\top z_\tau^{\rm G}(t)\in V_\tau\) solve the
semi-discrete Galerkin heat equation
\[
(\partial_t u_\tau^{\rm G}(t),v)_{L^2}+a(u_\tau^{\rm G}(t),v)=0,
\qquad v\in V_\tau,
\qquad
u_\tau^{\rm G}(0)=R_\tau u^\star(\cdot,0).
\]
Choose reference coefficients \(z_n^\dagger=z_\tau^{\rm G}(t_n)\).
Assume
\[
u^\star,\partial_tu^\star\in C([0,T];H^1_0(\Omega)).
\]
Define
\[
\varepsilon^{\rm G}_{M,\tau}(u^\star)
=
\sup_{0\le t\le T}
\|u^\star(\cdot,t)-R_\tau u^\star(\cdot,t)\|_{L^2(\Omega)}
+
\int_0^T
\|\partial_tu^\star(\cdot,t)-R_\tau\partial_tu^\star(\cdot,t)\|_{L^2(\Omega)}
\,dt .
\]
Then
\[
\sup_{0\le n\le N}
\|u^\star(\cdot,t_n)-\Psi^\top z_n^\dagger\|_{L^2(\Omega)}
\le \varepsilon^{\rm G}_{M,\tau}(u^\star).
\]
If \(u_\tau^{\rm G}\in C^2([0,T];L^2(\Omega))\) and
\[
M_{2,\tau}^{\rm G}
=\sup_{0\le t\le T}
\|\partial_{tt}u_\tau^{\rm G}(t)\|_{L^2(\Omega)}<\infty,
\]
then the implicit-Euler heat residual satisfies
\[
\|b_{n+1}\|_2\le \tfrac12 M_{2,\tau}^{\rm G}(\Delta t)^2,\qquad 0\le n<N,
\]
and hence
\[
\sup_{0\le n\le N}
\|\Psi^\top(z_n^\dagger-\bar z_n)\|_{L^2(\Omega)}
\le \tfrac12 M_{2,\tau}^{\rm G}T\,\Delta t .
\]
Consequently the contraction theorem applies to conforming effective-rank spaces with
the Galerkin approximation error in place of \(\varepsilon_{M,\tau}(u^\star)\),
\(p=1\), and no additional graph-domain generator-residual term. Equivalently, it verifies
\(\mathcal E_{\rm dyn}(\tfrac12M_{2,\tau}^{\rm G},1,1)\) for the Galerkin
reference. This is the theoretical support used for the Dirichlet-bubble
benchmark; uniform constants require a uniform bound on \(M_{2,\tau}^{\rm G}\).
\end{lemma}

This is the conforming route used later for the Dirichlet-bubble benchmark:
boundary compatibility moves the consistency burden from a graph-domain
residual to the Ritz approximation quantity and the Galerkin time-regularity
bound \(M_{2,\tau}^{\rm G}\). A proof sketch is given in
Appendix~\ref{app:aux-proofs}; the Supplementary Material contains the full
auxiliary proof details.

\subsubsection{Kalman Error}
The stochastic part of the final bound is now purely linear-Gaussian. The next
lemma says that, for data generated by the deterministic effective-rank
reference trajectory, the Kalman covariance controls the mean-square error of
the posterior mean in coefficient space.

\begin{lemma}[Kalman mean error for a deterministic effective-rank reference trajectory]
\label{lem:deterministic-truth-kalman-mean-error}
Condition on the mass-whitened feature draw and the observation matrices
\(\{H_n\}_{n=0}^N\). Let a deterministic effective-rank trajectory satisfy
\[
\bar z_{n+1}=F_n\bar z_n,\qquad 0\le n<N,
\]
and let the data be generated by
\[
y_n=H_n\bar z_n+\xi_n,\qquad
\xi_n\sim\mathcal N(0,\sigma^2I_{N_o}),
\]
with independent noises. Run the Kalman filter with the same forecast
matrices \(F_n\), observation matrices \(H_n\), and observation covariance
\(\sigma^2I_{N_o}\). Assume the initial forecast mean satisfies
\(m_0^f=\bar z_0\), the forecast covariances \(P_n^f\) are positive definite,
and the covariance recursion uses positive semidefinite process covariance
\(Q_n\succeq0\). Then, for every \(0\le n\le N\),
\[
\mathbb E_\xi\!\left[
\|m_n^a-\bar z_n\|_2^2
\;\middle|\;\{H_k\}_{k=0}^N
\right]
\le \operatorname{tr}P_n^a .
\]
\end{lemma}

\begin{proof}
Let \(e_n^f=m_n^f-\bar z_n\) and \(e_n^a=m_n^a-\bar z_n\). At time \(n\),
\[
e_n^a=(I-K_nH_n)e_n^f+K_n\xi_n,
\]
where
\[
K_n=P_n^fH_n^\top(H_nP_n^fH_n^\top+\sigma^2I)^{-1}.
\]
If \(\mathbb E_\xi[e_n^f(e_n^f)^\top\mid H]\preceq P_n^f\), then independence
of \(e_n^f\) and the current zero-mean noise \(\xi_n\), together with the
Joseph covariance identity, gives
\[
\mathbb E_\xi[e_n^a(e_n^a)^\top\mid H]
\preceq
(I-K_nH_n)P_n^f(I-K_nH_n)^\top+\sigma^2K_nK_n^\top
=P_n^a .
\]
The forecast error satisfies \(e_{n+1}^f=F_ne_n^a\), while the Kalman
forecast covariance is \(P_{n+1}^f=F_nP_n^aF_n^\top+Q_n\). Hence
\(\mathbb E_\xi[e_{n+1}^f(e_{n+1}^f)^\top\mid H]\preceq P_{n+1}^f\).
The induction starts from \(e_0^f=0\). Taking traces proves the claim.
\end{proof}

\begin{remark}[Initial bias]
\label{rem:initial-bias}
Lemma~\ref{lem:deterministic-truth-kalman-mean-error} is stated for the
matched initial forecast mean \(m_0^f=\bar z_0\), which isolates the
observation-noise contribution to the posterior mean error. If
\(m_0^f\neq \bar z_0\), the same recursion gives an additional propagated
initial-bias term. In particular, under a finite-horizon stability bound on
the forecast/update error propagators, the right-hand side is augmented by a
term of the form
\(\sup_{0\le n\le N}\|\mathcal E_{n:0}\|_2^2
\|m_0^f-\bar z_0\|_2^2\), where \(\mathcal E_{n:0}\) denotes the corresponding
forecast/update error propagator. We use the matched initialization in the
main theorem to keep the displayed contraction bound focused on feature
approximation, time consistency, and Bayesian estimation.
\end{remark}

The final estimate is obtained by adding three controlled quantities:
deterministic approximation of the heat trajectory by the effective-rank
space, deterministic dynamics bias from the chosen time integrator, and the
Kalman covariance of the coefficient posterior. These give the three terms
\(\varepsilon_{M,\tau}^2\), \((\Delta t)^{2p}\), and
\(r\sigma^2/[N_o(1-\eta)]\). All objects in the theorem are fixed in the
following setup.

Let \(t_n=n\Delta t\), \(0\le n\le N\), and \(T=N\Delta t\). Suppose the
mass-whitened feature draw satisfies Construction~\ref{ass:mass-whitened-rfm}, the
approximation event of Assumption~\ref{ass:retained-approximation}, and the
leverage event \(\mathcal E_{\rm lev}(L_\star)\). Let
\(\Psi=(\psi_1,\ldots,\psi_r)^\top\), so that
\[
\int_\Omega\Psi(x)\Psi(x)^\top\,d\mu_\Omega(x)=I_r,
\]
and choose \(z_n^\dagger\in\mathbb R^r\) such that
\[
\sup_{0\le n\le N}
\|u^\star(\cdot,t_n)-\Psi^\top z_n^\dagger\|_{L^2(\Omega)}
\le \varepsilon_{M,\tau}(u^\star).
\]
Let \(F_n\in\mathbb R^{r\times r}\) be the effective-rank forecast matrices, and
assume the dynamics input
\(\mathcal E_{\rm dyn}(C_b,p,S_T)\) of
Assumption~\ref{ass:retained-dynamics-consistency} holds for the same
reference trajectory \(\{z_n^\dagger\}_{n=0}^N\). For implicit-Euler heat with
\(\widetilde{\mathbb K}\succeq0\), the certificates in
Lemmas~\ref{lem:implicit-euler-heat-consistency}--\ref{lem:ritz-heat-consistency}
verify this input with \(S_T=1\) under their stated conditions.

Define the deterministic effective-rank reference trajectory
\[
\bar z_0=z_0^\dagger,\qquad \bar z_{n+1}=F_n\bar z_n .
\]
At each analysis time draw \(X_{n,1},\ldots,X_{n,N_o}\) independently from
\(\mu_\Omega\), independently across \(n\) and independently of the feature
draw, and set \((H_n)_{ia}=\psi_a(X_{n,i})\). The effective-rank observation model is
\[
y_n=H_n\bar z_n+\xi_n,\qquad
\xi_n\sim\mathcal N(0,\sigma^2I_{N_o}),
\]
with independent noises. Run the Kalman update with initial forecast mean
\(m_0^f=\bar z_0\), positive definite forecast covariances \(P_n^f\succ0\),
and positive semidefinite process covariances \(Q_n\succeq0\). Let
\(m_n^a\) and \(P_n^a\) be the analysis mean and covariance.
For implicit-Euler heat, \(P_n^f\succ0\) is preserved, for example, if
\(P_0^f\succ0\), \(Q_n\succeq0\), and
\(F_n=(I_r+\Delta t\,\widetilde{\mathbb K})^{-1}\), since each \(F_n\) is
invertible.

The theorem is therefore a stochastic contraction statement for the
coefficient model, with deterministic PDE consistency terms added through the
reference trajectory and the dynamics-bias certificate.

\subsubsection{Main Theorem}

\begin{maintheorem}[Posterior contraction and PDE consistency in mass-whitened RFM coordinates]
\label{thm:linear-heat-retained-contraction}
There is a universal constant \(C>0\) such that, if \(0<\eta<1\) and
\[
N_o\ge C L_\star\eta^{-2}\log\frac{2r(N+1)}{\delta_{\rm gram}},
\]
then, conditional on the mass-whitened feature draw,
\[
\max_{0\le n\le N}
\|N_o^{-1}H_n^\top H_n-I_r\|_2\le\eta
\]
with probability at least \(1-\delta_{\rm gram}\). On this event,
\[
\sup_{0\le n\le N}\operatorname{tr}P_n^a
\le \frac{r\sigma^2}{N_o(1-\eta)}
\]
and, for \(\widehat u_n=\Psi^\top m_n^a\),
\[
\sup_{0\le n\le N}
\mathbb E\!\left[
\|\widehat u_n-u^\star(\cdot,t_n)\|_{L^2(\Omega)}^2
\;\middle|\; \{H_k\}_{k=0}^N
\right]
\le
C_T\left[
\varepsilon_{M,\tau}(u^\star)^2
+(\Delta t)^{2p}
+\frac{r\sigma^2}{N_o(1-\eta)}
\right],
\]
where the expectation is over the observation noises \(\{\xi_k\}_{k=0}^N\).
The constant \(C_T\) depends on \(T\), \(S_T\), and \(C_b\), but not on
\(N_o\) or \(\sigma\); it is independent of \(r\) if \(C_b\) and \(S_T\) are
uniform over the effective-rank spaces. If the approximation and leverage events
have failure probabilities \(\delta_{\rm app}\) and \(\delta_{\rm lev}\), the
joint success probability over the feature draw and observation locations is at
least
\[
1-\delta_{\rm app}-\delta_{\rm lev}-\delta_{\rm gram}.
\]
\end{maintheorem}

\begin{proof}
The time-uniform Gram event follows directly from
Lemma~\ref{lem:time-uniform-retained-gram}. On this event,
\[
H_n^\top H_n\succeq N_o(1-\eta)I_r,\qquad 0\le n\le N.
\]
For any analysis time, the linear-Gaussian covariance update can be written in
information form as
\[
P_n^a=\left((P_n^f)^{-1}+\sigma^{-2}H_n^\top H_n\right)^{-1}.
\]
Since \(P_n^f\succ0\),
\[
P_n^a
\preceq
\left(\sigma^{-2}N_o(1-\eta)I_r\right)^{-1}
=\frac{\sigma^2}{N_o(1-\eta)}I_r.
\]
Taking traces and then the supremum over \(n\) proves the covariance
contraction bound.

By Lemma~\ref{lem:deterministic-retained-dynamics-bias},
\[
\sup_{0\le n\le N}
\left\|\Psi^\top(z_n^\dagger-\bar z_n)\right\|_{L^2(\Omega)}
\le
S_TC_bT\left(\varepsilon_{M,\tau}(u^\star)+(\Delta t)^p\right).
\]
The effective-rank approximation event gives
\[
\sup_{0\le n\le N}
\left\|u^\star(\cdot,t_n)-\Psi^\top z_n^\dagger\right\|_{L^2(\Omega)}
\le \varepsilon_{M,\tau}(u^\star).
\]
Lemma~\ref{lem:deterministic-truth-kalman-mean-error} and the
mass-whitening identity give
\[
\mathbb E_\xi\left[
\left\|\Psi^\top(m_n^a-\bar z_n)\right\|_{L^2(\Omega)}^2
\;\middle|\; \{H_k\}_{k=0}^N
\right]
\le \operatorname{tr}P_n^a.
\]
Combining the three errors by \((a+b+c)^2\le3(a^2+b^2+c^2)\) gives the
displayed risk bound after absorbing constants depending only on
\(T,S_T,C_b\) into \(C_T\). The joint high-probability statement is the union
bound over the approximation, leverage, and time-uniform Gram events.
\end{proof}

\begin{remark}[Per-time Gram versus accumulated information]
\label{rem:accumulated-observability}
The per-time empirical Gram lower bound in
Theorem~\ref{thm:linear-heat-retained-contraction} is a simple sufficient
condition, not a necessary condition for contraction. If \(N_o<r\), then
\(H_n^\top H_n\) cannot be full rank at a single analysis time. The covariance
update still satisfies the information-form identity
\[
P_n^a
=\left((P_n^f)^{-1}+\sigma^{-2}H_n^\top H_n\right)^{-1}
\]
whenever \(P_n^f\succ0\). Hence, for any deterministic lower bound
\[
\lambda_{\min}\!\left((P_n^f)^{-1}+\sigma^{-2}H_n^\top H_n\right)
\ge \alpha_n>0,
\]
one has \(\operatorname{tr}P_n^a\le r/\alpha_n\). In streaming regimes,
such lower bounds may come from accumulated observability over the filtering
window, equivalently from a stacked observability Gramian involving the
propagators and all past observation matrices. The theorem uses the stronger
per-analysis Gram event because it gives a transparent high-probability
sample-size condition.
\end{remark}

\begin{remark}[Deterministic bias convention]
\label{rem:model-error-convention}
The theorem keeps feature approximation and effective-rank dynamics residuals
as deterministic biases outside the Kalman state. The covariance trace is
therefore used in two roles: as the posterior variance in the coefficient
model, and as an upper bound for the mean-square error of the posterior mean
when the effective-rank reference trajectory generates the data. In
Theorem~\ref{thm:linear-heat-retained-contraction}, \(Q_n\) is allowed as an
algorithmic covariance inflation or prescribed process covariance in the
filter, while the reference trajectory generating the displayed risk bound is
deterministic. A fully stochastic model-error interpretation would require
additional \(Q_n\)-dependent stability terms
\cite{sanzalonso2015longtime,sanzalonso2025longtime}.
Thus the posterior dimension is \(r\), whereas the deterministic consistency
constant can still depend on the approximation family unless the Galerkin
time-regularity quantities are uniformly bounded.
\end{remark}

\subsection{Feature-family specializations}
\label{ssec:feature-corollaries}

Theorem~\ref{thm:linear-heat-retained-contraction} is feature-agnostic. A
feature family enters through the approximation error
\(\varepsilon_{M,\tau}(u^\star)\), the leverage envelope \(L_\star\), and the
dynamics input \(\mathcal E_{\rm dyn}(C_b,p,S_T)\). This subsection supplies
feature-level inputs for the first two quantities; the dynamics input remains
the PDE-consistency certificate from
Lemma~\ref{lem:implicit-euler-heat-consistency} or
Lemma~\ref{lem:ritz-heat-consistency}. Thus the corollaries below should be
read as checkable inputs to Theorem~\ref{thm:linear-heat-retained-contraction},
not as separate contraction theorems. For homogeneous Dirichlet heat, vanilla
RFFs and unwindowed ELU features are nonconforming and therefore require the
graph-domain route or a boundary-compatible modification.

In the approximation statements below, the discarded mass-eigendirection tail
has the following concrete meaning. If
\(u_{M,n}^{\rm MC}=\Phi^\top c_n^{\rm MC}\) is the Monte Carlo feature
approximant before the mass cutoff and \(P_\tau\) is the \(L^2\)-orthogonal
projector onto \(\operatorname{span}\Psi\), then
\[
\varepsilon_{\rm disc}(\tau)
=\sup_{0\le n\le N}
\|(I-P_\tau)u_{M,n}^{\rm MC}\|_{L^2(\Omega)}
=\sup_{0\le n\le N}
\left(\sum_{j\notin I_\tau}\lambda_j
\left|(V^\top c_n^{\rm MC})_j\right|^2\right)^{1/2}.
\]
The quantities \(\varepsilon_{\rm disc}^{\rm g}(\tau)\) and
\(\varepsilon_{\rm disc}^{\rm elu}(\tau)\) are this tail for the Gaussian-RFF
and ELU approximants, respectively.

\begin{corollary}[Gaussian RFF approximation and leverage certificate]
\label{cor:gaussian-rff-specialization}
On \(\Omega=[0,1]^d\), let
\(\varphi_j(x)=\sqrt 2\cos(\omega_j\cdot x+b_j)\) with
\(\omega_j\stackrel{\rm iid}{\sim}\mathcal N(0,\gamma^2I_d)\) and
\(b_j\sim\operatorname{Unif}(0,2\pi)\). The associated kernel is
the squared exponential
\(k(x,y)=\exp(-\gamma^2\|x-y\|_2^2/2)\), and the Gaussian-RFF
integral class \(\mathcal F_{\rm g}(C_u)\) consists of targets
representable as
\(u(x)=\int \sqrt 2\cos(\omega\cdot x+b)\,a(\omega,b)\,d\pi(\omega,b)\)
with \(\|a\|_\infty\le C_u\)
\cite{rahimi2008random,bach2017equivalence}. If
\(u^\star(\cdot,t_n)\in\mathcal F_{\rm g}(C_u)\) uniformly in
\(n\), then with probability at least
\(1-\delta_{\rm app}\) the effective-rank approximation event of
Assumption~\ref{ass:retained-approximation} holds with
\[
\varepsilon_{M,\tau}(u^\star)
\le
\sqrt 2\,C_u\sqrt{C_{\rm RF}\,M^{-1}\log\bigl(2(N+1)/\delta_{\rm app}\bigr)}
+\varepsilon_{\rm disc}^{\rm g}(\tau),
\]
where \(C_{\rm RF}\) is a universal constant and
\(\varepsilon_{\rm disc}^{\rm g}(\tau)\) is the discarded
mass-eigendirection tail of the Monte Carlo approximant. The grid-certificate
leverage envelope in the Supplementary Material applies with the global
raw bounds \(E_\Phi=2M\) and
\(G_\Phi=2\sum_j\|\omega_j\|_2^2\), the latter controlled with
probability \(1-\delta_{\rm grad}\) by the Laurent-Massart
\(\chi^2\) tail
\(G_\Phi\le 2\gamma^2[Md+2\sqrt{Md\log(1/\delta_{\rm grad})}+2\log(1/\delta_{\rm grad})]\).
Analytic certification of \(L_\star\) through this gradient bound carries the
additional failure probability \(\delta_{\rm grad}\).
The feature-certificate failure probability is therefore at most
\(\delta_{\rm app}+\delta_{\rm grad}\), before the Gram event is added.
If either the graph-residual certificate of
Lemma~\ref{lem:implicit-euler-heat-consistency} or the conforming Galerkin
condition of Lemma~\ref{lem:ritz-heat-consistency} is also verified, these
inputs satisfy the hypotheses of
Theorem~\ref{thm:linear-heat-retained-contraction}.
\end{corollary}

For Gaussian RFFs, the Monte Carlo approximation controls
\(\varepsilon_{M,\tau}\), while the gradient envelope controls \(L_\star\).

\begin{corollary}[Conditional ELU/RFM approximation and leverage certificate]
\label{cor:elu-rfm-specialization}
On \(\Omega=[0,1]^d\), let
\(\varphi_j(x)=s\,\sigma_\alpha(w_j\cdot x+b_j)\) with the ELU
activation
\(\sigma_\alpha(z)=z\cdot \mathbf 1\{z>0\}+\alpha(e^z-1)\cdot\mathbf 1\{z\le 0\}\)
and frozen \((w_j,b_j)\). Under the standard single-hidden-layer
RFM approximation hypothesis
\cite{chen2022bridging,chen2023timedependent,ming2025spectral},
\[
\varepsilon_{M,\tau}(u^\star)
\le
C_{\rm elu}(u^\star,T)
\sqrt{\log(2(N+1)/\delta_{\rm app})/M}
+\varepsilon_{\rm disc}^{\rm elu}(\tau)
\]
with probability at least \(1-\delta_{\rm app}\), where
\(\varepsilon_{\rm disc}^{\rm elu}(\tau)\) is the analogous discarded
mass-eigendirection tail.
The grid-certificate leverage envelope in the Supplementary Material applies with
\[
E_\Phi=B_{\rm elu}^2,\qquad
G_\Phi=s^2\max\{1,\alpha\}^2\sum_j\|w_j\|_2^2,
\]
where
\[
B_{\rm elu}^2
=s^2\sum_j\operatorname*{ess\,sup}_x
|\sigma_\alpha(w_j\cdot x+b_j)|^2;
\]
under the sampling convention
\(w_j\sim\mathcal N(0,a^2d^{-1}I_d)\),
\(b_j=-w_j\cdot c_j+\xi_j\), with centers \(c_j\in\Omega\) drawn
independently from the design distribution and
\(\xi_j\sim\mathcal N(0,\beta^2)\), a
union bound on Gaussian concentration gives, with probability at
least \(1-\delta_{\rm env}\),
\[
B_{\rm elu}\le
sM^{1/2}\Bigl[
\alpha+a\bigl(\sqrt d+\sqrt{2\log(2M/\delta_{\rm env})}\bigr)
+\beta\sqrt{2\log(2M/\delta_{\rm env})}
\Bigr].
\]
Analytic certification of \(L_\star\) through this envelope carries the
additional failure probability \(\delta_{\rm env}\).
The feature-certificate failure probability is therefore at most
\(\delta_{\rm app}+\delta_{\rm env}\), before the Gram event is added.
With the dynamics input \(\mathcal E_{\rm dyn}(C_b,p,S_T)\), these estimates
satisfy the hypotheses of
Theorem~\ref{thm:linear-heat-retained-contraction}. If the ELU features are
multiplied by a smooth zero-trace cutoff, Lemma~\ref{lem:ritz-heat-consistency}
verifies that input through the conforming Galerkin route.
\end{corollary}

For shallow ELU features, the approximation statement is the RFM input and the
displayed envelopes provide the leverage certificate needed for the Gram bound.

\begin{corollary}[Dirichlet-bubble Gaussian RFF certificate]
\label{cor:bubble-rff-specialization}
On \(\Omega=[0,1]^d\), let
\(\varphi_j(x)=b(x)\sqrt 2\cos(\omega_j\cdot x+b_j)\) with the
bubble \(b(x)=\prod_i x_i(1-x_i)\) and Gaussian
\((\omega_j,b_j)\) as in
Corollary~\ref{cor:gaussian-rff-specialization}. Every feature
satisfies \(\varphi_j|_{\partial\Omega}=0\), so the Galerkin
discretization of the homogeneous-Dirichlet heat operator on
\(H^1_0(\Omega)\) is conforming. For targets in the windowed class
\(\mathcal F_{\rm g,b}(C_u)=\{b\cdot v:v\in\mathcal F_{\rm g}(C_u)\}\),
the approximation bound of
Corollary~\ref{cor:gaussian-rff-specialization} carries through
with the additional multiplicative factor
\(\|b\|_{L^\infty(\Omega)}=4^{-d}\)
(\(\|b\|_{L^\infty}=1/16\) for the 2D experiments of
Section~\ref{sec:numerics}). The leverage envelope and Gram
concentration events inherit from
Corollary~\ref{cor:gaussian-rff-specialization} with a
deterministic Lipschitz adjustment for the bubble factor.
For this conforming space, if the parabolic Ritz quantity
\[
\varepsilon^{\rm G}_{M,\tau}(u^\star)
=
\sup_{0\le t\le T}
\|u^\star(\cdot,t)-R_\tau u^\star(\cdot,t)\|_{L^2(\Omega)}
+
\int_0^T
\|\partial_tu^\star(\cdot,t)-R_\tau\partial_tu^\star(\cdot,t)\|_{L^2(\Omega)}
\,dt
\]
is small on the bubble space and \(M_{2,\tau}^{\rm G}\) is bounded,
Lemma~\ref{lem:ritz-heat-consistency} verifies the dynamics input with
\(C_b=\tfrac12M_{2,\tau}^{\rm G}\), \(p=1\), and \(S_T=1\). This conforming
Galerkin certificate is the route used for the Dirichlet-bubble benchmark.
Applying Theorem~\ref{thm:linear-heat-retained-contraction} gives the same
contraction bound. Its approximation term is
\(\varepsilon^{\rm G}_{M,\tau}(u^\star)\), its estimation term is
\(r\sigma^2/[N_o(1-\eta)]\), and its constants depend on the Galerkin
time-regularity bound. The windowed integral-class bound is a sufficient RFF
approximation certificate; for general homogeneous-Dirichlet targets the
approximation input is the Ritz quantity
\(\varepsilon^{\rm G}_{M,\tau}(u^\star)\), which is reported in
the benchmark.
\end{corollary}

This is the boundary-compatible route used in the two-dimensional heat
benchmark: the bubble factor enforces the homogeneous Dirichlet trace, so the
conforming Ritz certificate is the relevant dynamics input.

The deterministic sine-basis identity used in the one-dimensional experiment
is recorded in the Supplementary Material. The proofs of the feature-family
corollaries are also given there.

\medskip

Together, the effective coordinates, proof certificates, and
heat-consistency estimates yield a quantitative posterior-contraction bound
for the linear heat equation. The Supplementary Material records conditional
analogues for Caputo subdiffusion and a linearized Darcy inverse problem. The
next section verifies the terms in the bound numerically and compares the
resulting filter with the ensemble-transform Kalman filter.

\section{Numerical Experiments}
\label{sec:numerics}

The experiments are organized around one question: does the coefficient-space
posterior analysis in Theorem~\ref{thm:linear-heat-retained-contraction}
predict the numerical behavior that matters in streaming PDE data
assimilation? We first isolate the three terms in the theorem,
effective-rank feature approximation, time-discretization bias, and the
\(r\sigma^2/N_o\) Bayesian-estimation term. We then test the resulting
filter in an under-observed two-dimensional heat benchmark, assess its
posterior calibration and \(N_o\)-scaling, and show why boundary-compatible
random features are needed to control deterministic PDE bias. A final
non-self-adjoint advection-diffusion benchmark tests the same exact
coefficient-space filtering mechanism beyond the self-adjoint heat theorem.
The Supplementary Material reports additional linear-Gaussian variations and
supporting numerical details.
Throughout this section, \(N_t\) denotes the number of analysis times in
a benchmark; the theory uses \(N\) for the same role.

\subsection{Three error mechanisms}
\label{sec:exp-verification}

We begin by separating the mechanisms that are coupled in the full streaming
benchmark. In an orthonormal sine basis, the approximation tail, the
implicit-Euler bias, and the Kalman variance can be varied one at a time. This
establishes the expected slopes before non-orthogonal random features and
random observation geometry are introduced.

\paragraph{Rates on the 1D heat equation}
The 1D heat equation in the orthonormal sine basis is the cleanest
testbed: the empirical Gram is exact by the standard DST orthogonality
identity, so the three terms in the contraction bound
can be exercised one at a time. Sweeping \(M\) with all other
parameters fixed shows that the approximation term takes on the
regularity of the data rather than a worst-case Sobolev rate; sweeping
\(\Delta t\) recovers the predicted \(p=1\) order of implicit Euler;
and sweeping \(N_o\) and \(\sigma\) recovers the predicted linear scaling
of the posterior variance with the observation-noise-to-information
ratio. The measured slopes confirm the predicted scalings in this controlled
regime.

\begin{figure}[t]
\centering
\begin{minipage}{0.327\linewidth}
\includegraphics[width=\linewidth]{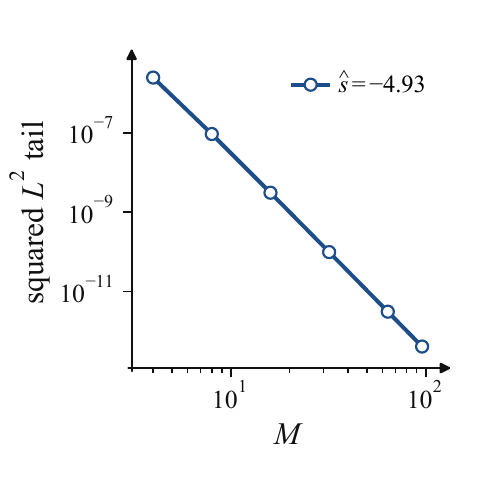}
\end{minipage}\hspace{0.006\linewidth}%
\begin{minipage}{0.327\linewidth}
\includegraphics[width=\linewidth]{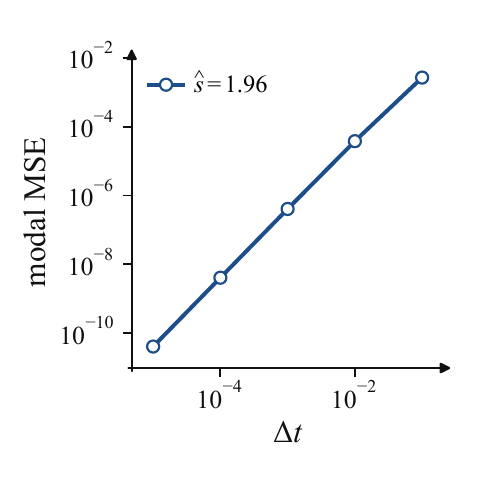}
\end{minipage}\hspace{0.006\linewidth}%
\begin{minipage}{0.327\linewidth}
\includegraphics[width=\linewidth]{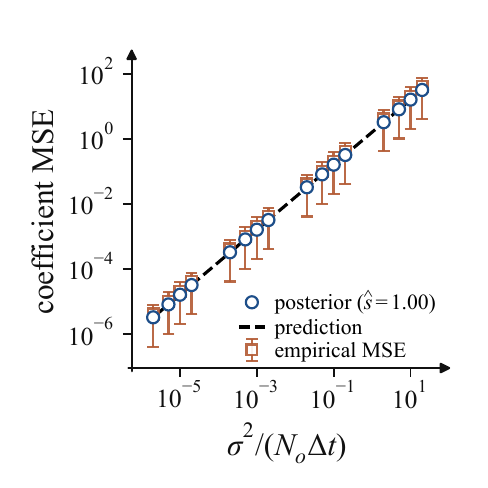}
\end{minipage}
\caption{The three predicted scalings on the 1D heat equation. Left:
sine-basis truncation tail, slope \(-4.93\). Center: implicit-Euler modal
mean-square error, slope \(1.96\). Right: streaming Kalman variance versus
\(\sigma^2/(N_o\Delta t)\), slope \(0.9998\).}
\label{fig:heat-rates}
\end{figure}

\paragraph{Mass whitening for non-orthogonal random features}
The sine basis is exceptional because both its continuum mass matrix and its
empirical Gram are the identity. For ELU and Gaussian RFF features,
Construction~\ref{ass:mass-whitened-rfm} replaces the raw coefficient state by
a mass-orthonormal effective coordinate. Across representative floor-selected
draws, the empirical Gram fluctuation has fitted slopes between \(-0.46\) and
\(-0.61\), while the heat-Kalman posterior trace has slopes between \(-1.01\)
and \(-1.12\). Thus the random-feature experiments recover the same
\(N_o^{-1/2}\) Gram and \(N_o^{-1}\) posterior-trace scalings after mass
whitening. The floor-selection values are reported in the Supplementary
Material.

The theorem treats \(\tau\) as fixed before the probability statement. When
\(\tau\) is selected from a finite grid, a union bound over the
candidate floors gives the corresponding finite-grid certificate. In the
experiments this selection is used as a numerical conditioning check.

The matrix-Chernoff condition is a transparent sufficient condition for the
theorem. It is not intended as a sharp experimental threshold. For the
reference \(M=80\), \(\tau=10^{-6}\) draw used below, the realized
dense-grid leverage is about \(875\), while the certified envelope is
\(5.3\times10^5\). The observed \(1/N_o\) regime begins at much smaller
observation counts because the realized leverage and accumulated filtering
information are substantially better than the worst-case certificate.
Quadrature assembly was
checked against order-128 Gauss--Legendre
tensor quadrature; the order-24 rule used below gives relative Frobenius
errors below \(10^{-13}\) in both mass and stiffness matrices for the
reference frequency scale.

\subsection{Under-observed streaming heat}
\label{sec:exp-benchmark}
\label{sec:heat-enkf-benchmark}

This benchmark is intentionally in the practically important under-observed
regime \(N_o<r\). It is therefore outside the per-time Gram condition used in
Theorem~\ref{thm:linear-heat-retained-contraction}, but it probes the
accumulated-observability mechanism described in
Remark~\ref{rem:accumulated-observability}: each analysis time is
under-observed, while information is accumulated through the filtering dynamics
over the observation window. The benchmark uses the 2D heat equation on
\([0,1]^2\), \(M=120\) Dirichlet-bubble Gaussian RFFs, effective rank
\(r=61\), \(N_t=20\) analysis times, \(N_o=24\) random interior observations
per time, and \(\sigma=10^{-2}\). Observations are generated from the
effective-rank heat evolution of a projected out-of-prior two-bump initial
field; the Supplementary Material reports the corresponding unprojected-truth
comparison.

The comparison isolates posterior computation. RFM-Kalman and the offline
MAP+Laplace solve compute the same effective-rank Gaussian posterior, while
ETKF approximates that posterior by a deterministic square-root ensemble with
fixed inflation \(\rho=1.05\).

\begin{table}[t]
\centering
\footnotesize
\setlength{\tabcolsep}{4pt}
\begin{tabular}{lrrr}
\toprule
Method (Dirichlet-bubble Gaussian RFF, \(r=61\)) & Field \(L^2\) error & Trace & Wall time \\
\midrule
RFM-Kalman   & \((3.44\pm 1.0)\!\times\!10^{-4}\) & \(1.84\!\times\!10^{-7}\) & \(3.06~\text{ms}\) \\
MAP+Laplace  & \((3.44\pm 1.0)\!\times\!10^{-4}\)         & \(1.84\!\times\!10^{-7}\) & \(0.83~\text{ms}\) \\
\midrule
\multicolumn{4}{l}{\emph{ETKF with multiplicative inflation \(\rho=1.05\):}} \\
ETKF, \(N_e=50\)   & \((1.27\pm 0.52)\!\times\!10^{-3}\) & \(4.84\!\times\!10^{-7}\) & \(6.13~\text{ms}\) \\
ETKF, \(N_e=100\)  & \((3.82\pm 0.96)\!\times\!10^{-4}\) & \(5.09\!\times\!10^{-7}\) & \(15.1~\text{ms}\) \\
ETKF, \(N_e=200\)  & \((3.82\pm 0.96)\!\times\!10^{-4}\) & \(5.09\!\times\!10^{-7}\) & \(64.3~\text{ms}\) \\
ETKF, \(N_e=500\)  & \((3.82\pm 0.96)\!\times\!10^{-4}\) & \(5.09\!\times\!10^{-7}\) & \(464~\text{ms}\) \\
ETKF, \(N_e=1000\) & \((3.82\pm 0.96)\!\times\!10^{-4}\) & \(5.09\!\times\!10^{-7}\) & \(2.31~\text{s}\) \\
\bottomrule
\end{tabular}
\caption{Streaming 2D heat-DA benchmark in the under-observed regime
\(N_o<r\). RFM-Kalman and MAP+Laplace use 16 truth seeds; each ETKF row uses
16 truth seeds and 3 ensemble seeds per truth. Times are solver-core wall
times for the 20-step filtering sequence.}
\label{tab:heat-enkf-bench}
\end{table}

For \(N_e\ge100\), larger ensembles change the reported mean error only below
the displayed precision. The visible plateau is therefore a full-rank ensemble
effect for this coefficient state, not an indication that the ETKF cost has
stopped increasing.

\begin{figure}[t]
\centering
\includegraphics[width=0.92\linewidth]{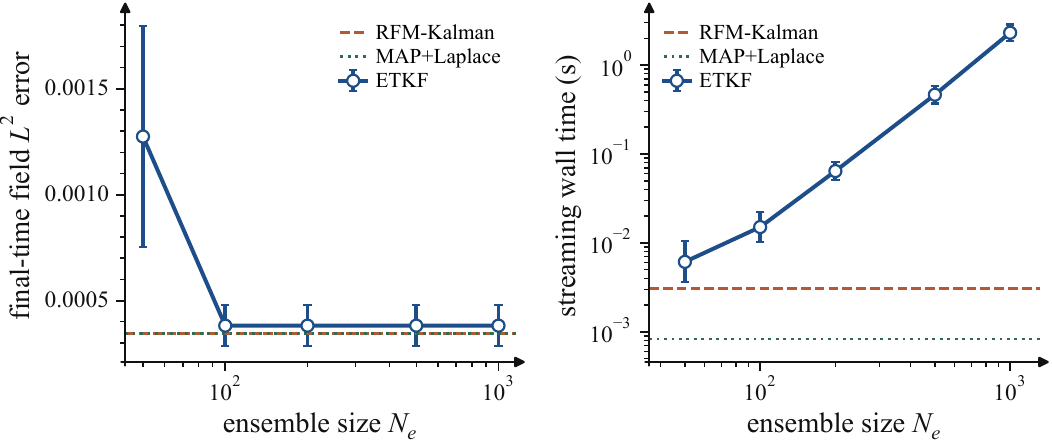}
\caption{Streaming 2D heat-DA benchmark. Left: final-time field
\(L^2\)-error. Right: solver-core wall time for the 20-step sequence,
excluding common basis assembly and observation generation. ETKF cost grows
with \(N_e\), while RFM-Kalman and the batched MAP+Laplace check compute the
effective-rank linear-Gaussian posterior at fixed rank. The benchmark is
intentionally outside the per-time full-rank Gram condition and tests
accumulated observability over the filtering window.}
\label{fig:heat-enkf-bench}
\end{figure}

The timing comparison separates online filtering from the offline batched
regression check. MAP+Laplace is the fastest core solve here because it solves
one stacked linear-Gaussian problem; RFM-Kalman maintains the filtering state at
each observation time and costs \(O(N_or^2+r^3)\) per analysis step. ETKF pays
for ensemble propagation and analysis transforms. Times were measured with
Python \texttt{time.perf\_counter} in the same script on an Apple M3 Pro
(12 cores), macOS 15.0.1, and Python 3.12.2. This benchmark isolates the
linear heat regime in which the effective-rank linear-Gaussian model is closed form.

\subsection{Posterior uncertainty and information scaling}
\label{sec:calibration}

A Bayesian estimate must be judged by both its posterior mean and its
uncertainty. After the mean-error comparison above, we therefore ask whether
the recovered coefficient posterior has the right field-level uncertainty
scale. We evaluate pointwise calibration on the same
Gauss--Legendre quadrature grid used to assemble the continuum
matrices, and compute pointwise standardized residuals
\[
Z_k=\frac{u^{\rm true}_{T,k}-\mathbb E[u_T(x_k)]}
{\sqrt{\operatorname{Var}[u_T(x_k)]}}.
\]
For each truth seed we compute the pointwise empirical coverage fraction over
quadrature nodes. RFM-Kalman gives mildly conservative pointwise coverage:
the mean coverage fractions over truth seeds are \(57.6\%\), \(94.0\%\), and
\(100.0\%\) for nominal \(50\%\), \(90\%\), and \(99\%\) intervals,
respectively. MAP+Laplace gives the same calibration up to sampling error,
as expected from the shared Gaussian posterior. ETKF coverage is more
conservative under the fixed inflation \(\rho=1.05\).

\subsubsection{\texorpdfstring{\(N_o\)-scaling}{N\_o-scaling}: the \texorpdfstring{\(1/N_o\)}{1/N\_o} rate}
\label{sec:sigma-no-sweep}

The coefficient-space covariance mechanism in
Theorem~\ref{thm:linear-heat-retained-contraction} predicts posterior trace
decay proportional to \(r\sigma^2/N_o\) at fixed \(r\) and \(\sigma\). We sweep
\(\sigma\in\{10^{-3},10^{-2},10^{-1}\}\) and
\(N_o\in\{6,12,24,48,96\}\) on the reference benchmark, then extend
the \(\sigma=10^{-2}\) sweep to \(N_o=768\). The pre-asymptotic slopes
are steeper than \(-1\), while the high-information tail recovers the same
\(N_o^{-1}\) information scaling as the coefficient-space covariance
mechanism in Theorem~\ref{thm:linear-heat-retained-contraction}.
In the extended tail \(N_o\in\{96,192,384,768\}\) at \(\sigma=10^{-2}\), the
posterior-trace slopes are \(-1.024\) for RFM-Kalman and \(-1.021\) for ETKF,
both with \(R^2\) essentially one. Thus the numerical information scaling
matches the \(1/N_o\) covariance mechanism, although this benchmark remains
outside the theorem's per-time Gram condition and is interpreted through
Remark~\ref{rem:accumulated-observability}. The Supplementary Material records
the corresponding pre-asymptotic slope ranges.

\subsection{Boundary compatibility controls deterministic bias}
\label{sec:exp-basis-design}
\label{sec:basis-comparison}

On a rectangle with homogeneous Dirichlet conditions, a deterministic
sine basis is the natural spectral baseline. We therefore compare
three effective-rank states at \(r\approx60\): vanilla cosine Gaussian
RFFs, Dirichlet-bubble Gaussian RFFs, and the 2D Dirichlet sine basis.
The bubble factor closes the boundary-condition mismatch of the
vanilla RFF basis and makes the random-feature state competitive with
the spectral basis at matched rank. This experiment isolates deterministic
bias rather than searching for the fastest basis: it tests whether the feature
space respects the PDE boundary geometry strongly enough for the Bayesian
estimation term to dominate.

\begin{table}[t]
\centering
\footnotesize
\setlength{\tabcolsep}{4pt}
\begin{tabular}{lccc}
\toprule
Basis & In-basis residual & Field \(L^2\) error & Posterior trace \\
\midrule
Vanilla Gaussian RFF (\(r{=}60\))      & \(2.4\!\times\!10^{-2}\) & \((9.0\pm 2.9)\!\times\!10^{-4}\) & \(1.0\!\times\!10^{-6}\) \\
Bubble Gaussian RFF (\(r{=}60\))       & \(9.6\!\times\!10^{-3}\) & \((3.4\pm 1.1)\!\times\!10^{-4}\) & \(1.8\!\times\!10^{-7}\) \\
2D Dirichlet sine (\(r{=}60\))         & \(5.1\!\times\!10^{-3}\) & \((3.5\pm 1.3)\!\times\!10^{-4}\) & \(1.9\!\times\!10^{-7}\) \\
\bottomrule
\end{tabular}
\caption{Streaming heat-DA at matched effective rank and observation design.
The vanilla and sine rows are paired on identical observation streams; the
bubble row uses the boundary-compatible RFF setting used in the benchmark. The
Dirichlet-bubble basis and the sine basis agree at the few-percent level
in field error and posterior trace, while the vanilla cosine RFF basis is
approximately \(2.6\times\) worse than the sine basis in field error.}
\label{tab:basis-comparison}
\end{table}

\paragraph{Unprojected-truth comparison}
\label{sec:unprojected-truth}
The main benchmark error is measured against the effective-rank in-basis
truth so that RFM-Kalman, MAP+Laplace, and ETKF are compared on the
same coefficient state. A separate unprojected-truth comparison,
reported in the Supplementary Material, shows that the bubble basis and
the sine basis have the same final-time error against a high-resolution
Dirichlet-sine reference, while the vanilla cosine RFF basis incurs a
\(7\times10^{-2}\) deterministic dynamics-bias plateau from its natural
(Neumann-like) boundary condition. This is precisely the bias term
isolated in Lemma~\ref{lem:deterministic-retained-dynamics-bias}. In the
boundary-compatible bubble-RFF run the corresponding dynamics-bias term is
\(5.7\times10^{-6}\), and the implicit-Euler Ritz/Galerkin analogue from
Lemma~\ref{lem:ritz-heat-consistency} is \(1.2\times10^{-5}\), both below the
coefficient posterior-estimation error. The supplement also reports the continuous-time
parabolic Ritz quantity \(\varepsilon^{\rm G}_{M,\tau}\), which includes the
time-derivative approximation entering the theorem. These values certify
the dynamics-consistency input for the benchmark. The RFF construction is valuable
because the same effective-rank random-feature procedure can be paired with a
problem-specific zero-trace cutoff on domains or operators without analytic
eigenfunctions.

\subsection{Beyond self-adjoint heat}
\label{sec:advdiff-benchmark}

The heat contraction theorem treats the self-adjoint parabolic case.
The coefficient-space construction itself is broader: after the random
features are frozen, any fixed linear-Gaussian Galerkin model has an exact
Kalman posterior. Having completed the heat-equation checks, we test this
mechanism in a more demanding linear PDE by repeating the streaming benchmark
for
\[
\partial_t u+b\cdot\nabla u=\nu\Delta u,\qquad u|_{\partial\Omega}=0,
\]
on \([0,1]^2\), with \(b=(4,2)\), \(\nu=1\), and P\'eclet number
\(\|b\|/\nu\approx4.47\). The advection term makes the assembled generator
non-self-adjoint, so the sine basis no longer diagonalizes the dynamics. This
experiment is outside the self-adjoint heat theorem, but remains inside the
exact frozen-feature linear-Gaussian coefficient framework. As a diagnostic,
the implicit-Euler propagator has spectral radius \(0.953\); a theorem-level
extension in this non-normal setting would require a spectral-norm or
finite-horizon product bound.

We use \(M=80\) Dirichlet-bubble Gaussian RFFs, mass floor
\(\tau=10^{-6}\), effective rank \(r=51\), \(N_t=20\) analysis times,
\(N_o=24<r\) random observations per time, and \(\sigma=10^{-2}\). Across 16
truth seeds, RFM-Kalman and the offline MAP+Laplace solve again agree to
roundoff at the posterior level. The RFM-Kalman final-time field error is
\((3.54\pm1.04)\times10^{-4}\), with trace \(1.67\times10^{-7}\) and
solver-core time \(2.39\) ms. The largest ETKF run, \(N_e=1000\), gives
\((3.89\pm1.05)\times10^{-4}\), trace \(4.72\times10^{-7}\), and
\(2.35\) s core time. Thus the exact coefficient-space recursion retains the
same posterior advantage in a non-normal streaming regime, while ETKF pays the
ensemble propagation cost.

\begin{figure}[t]
\centering
\includegraphics[width=0.92\linewidth]{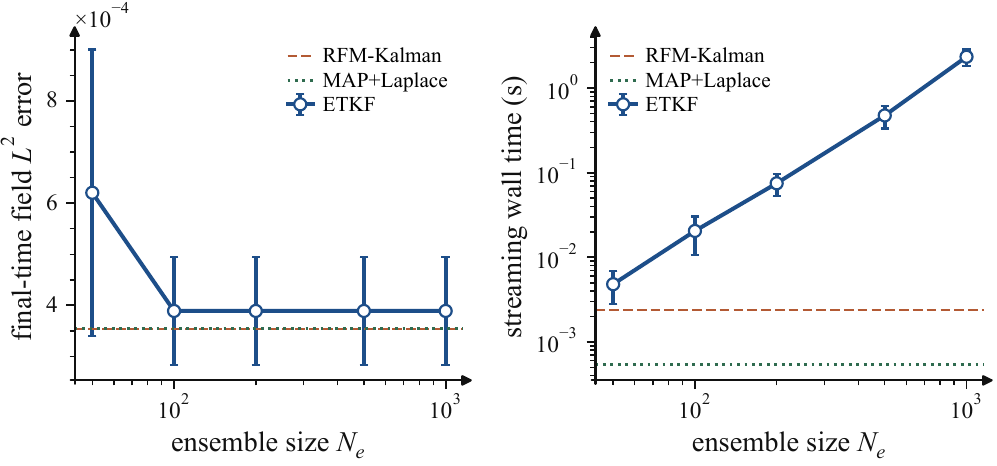}
\caption{Streaming 2D advection-diffusion benchmark with non-self-adjoint
generator at P\'eclet number \(\approx4.47\). Left: final-time field
\(L^2\)-error. Right: solver-core wall time for the 20-step sequence.
RFM-Kalman and MAP+Laplace compute the same effective-rank Gaussian posterior;
ETKF approaches the same error level only with large ensembles and substantially
higher core time.}
\label{fig:advdiff-benchmark}
\end{figure}

\section{Discussion}
\label{sec:discussion}

The framework rests on one structural observation: after the random features
are frozen, linear PDE data assimilation becomes a linear-Gaussian filtering
problem in coefficient space. The Kalman recursion is therefore the exact
posterior for the chosen coefficient model, while mass-whitened effective-rank
coordinates make non-orthogonal feature draws analytically usable. For heat,
the contraction theorem separates coefficient-space posterior uncertainty
from deterministic feature approximation and time-discretization errors; the
experiments exercise the same decomposition, check posterior calibration, and
show that a Dirichlet-bubble Gaussian RFF basis reaches the same \(10^{-4}\)
field-error level as a matched Dirichlet-sine basis.

The theorem is a finite-horizon result for homogeneous heat with Gaussian
noise. Its assumptions are certificates: implicit-Euler stability,
effective-rank dynamics consistency, and empirical Gram concentration for
i.i.d.\ uniform sensors. The per-time Gram condition is a sufficient
high-information certificate, not a necessary observability condition; the
two-dimensional benchmarks probe the practically important \(N_o<r\) regime
where information accumulates through the filtering window. The
advection-diffusion experiment confirms the same exact coefficient posterior
for a non-self-adjoint linear generator, but a contraction theorem in that
setting would require a propagator-adapted stability and observability
certificate.

\paragraph{Noise level, SNR, and non-Gaussian observations}
The Gaussian observation model is used here because it makes the frozen-feature
coefficient model conjugate, so that the Kalman recursion is the exact Bayesian
posterior and the covariance trace has the interpretation used in the theorem.
Lower-SNR regimes are not excluded; they increase the Bayesian-estimation term
unless one adds observations, lengthens the assimilation window, lowers the
effective rank, or uses a more informative prior. With non-Gaussian errors, the
same linear coefficient model still supports Kalman-type second-moment or
approximate filters for mean-zero finite-covariance noise, but exact Gaussian
posteriority and the present contraction proof would require additional robust
or non-Gaussian analysis.

The main implication is that random-feature PDE filtering can be treated as a
closed-form coefficient-space Bayesian update with effective dimension \(r\),
not as an ensemble approximation or a MAP-only regression. The online update
still contains explicit \(N_o\)-dependent linear algebra, but its state
dimension is set by the effective feature rank rather than by sensors or
collocation points. Natural extensions include stochastic model-error
formulations for the effective-rank truncation and iterated
linear-Gaussian updates for nonlinear PDEs, using the same mass-whitened
coefficient geometry.

\appendix
\section{Auxiliary Proofs}
\label{app:aux-proofs}

\paragraph{Proof of Lemma~\ref{lem:time-uniform-retained-gram}}
Condition on the mass-whitened feature draw. For a fixed analysis time,
Lemma~\ref{lem:retained-gram-concentration} gives
\[
\Pr\left[
\left\|N_o^{-1}H_n^\top H_n-I_r\right\|_2>\eta
\;\middle|\;\Psi
\right]
\le \frac{\delta_{\rm gram}}{N+1}
\]
provided
\[
N_o\ge C L_\star\eta^{-2}
\log\frac{2r(N+1)}{\delta_{\rm gram}},
\]
after increasing the universal constant \(C\) if necessary. A union bound
over \(n=0,\ldots,N\) gives the stated time-uniform event. If the same
observation locations are reused at every analysis time, then all \(H_n\) are
the same matrix in the time-independent effective basis, so the single-time
Gram bound already gives the uniform statement with
\(\log(2r/\delta_{\rm gram})\).

\paragraph{Proof sketch for Lemmas~\ref{lem:implicit-euler-heat-consistency}
and~\ref{lem:ritz-heat-consistency}}
Both estimates are deterministic consistency inputs for
Theorem~\ref{thm:linear-heat-retained-contraction}. In the graph-residual
route, one applies backward Taylor expansion to the projected heat trajectory,
uses \(L^2\)-orthonormality of the effective features, and bounds the generator
commutator by the assumed graph-domain residual, giving an
\(O(\Delta t\,\varepsilon_{M,\tau})+O((\Delta t)^2)\) one-step residual. In the
conforming route, Ritz orthogonality gives the parabolic Galerkin
approximation bound, and backward Taylor expansion of the Galerkin coefficient
trajectory gives the \(O((\Delta t)^2)\) residual. The Supplementary Material
contains the full auxiliary proof details and the remaining feature-certificate
arguments.

\bibliographystyle{siam}
\bibliography{references}

\end{document}